\documentclass[11pt,reqno]{amsart}

\usepackage{comment}
\usepackage{ragged2e}
\usepackage{wrapfig}
\usepackage{multimedia}
\usepackage{media9}
\usepackage{pst-node}
\usepackage{mathtools}
\usepackage{amsmath,amssymb,amsthm,mathrsfs}
\usepackage{xfrac}
\usepackage{amsfonts}
\usepackage{tikz-cd} 
\usepackage[english]{babel}
\usepackage[utf8x]{inputenc}
\usepackage{ytableau}

\newcounter{theoremc}
\newcounter{corollc}
\newcounter{propc}

\newtheorem{theorem}[theoremc]{Theorem}
\newtheorem*{theorem*}{Theorem}
\newtheorem*{defn*}{Definition}
\newtheorem*{prob*}{Problem}
\newtheorem*{lemma*}{Lemma}
\newtheorem*{conj*}{Conjecture}
\newtheorem*{corol*}{Corollary}
\newtheorem*{conclusion*}{Conclusion}
\newtheorem*{proposition*}{Proposition}

\newtheorem{proposition}[propc]{Proposition}
\newtheorem{corollary}[corollc]{Corollary}

\newtheorem*{remark*}{Remark}

\title{Littlewood-Richardson coefficients as a signed sum of Kostka numbers}

\author[S.~Shrivastava]{Sagar Shrivastava}
\address{School of mathematics, 
Tata Institute of Fundamental Research, 
Homi Bhabha Road, Mumbai 400005, India}
\email{sagars@math.tifr.res.in}

\date{}

\begin{document}

\maketitle

\begin{abstract}
Littlewood-Richardson (LR) coefficients and Kostka Numbers appear in representation theory and combinatorics related to $GL_n$. It is known that Kostka numbers can be represented as special Littlewood-Rischardson coefficient \cite{king2004stretched}. In this paper, we show how one can represent LR coefficient as a signed sum of Kostka numbers, and use the formulation to give a polynomial time algorithm for the same, hence showing that they belong to the same class of decision problems. As a corollary, we will prove Steinberg's formula \cite{steinberg1961general} using Kostant's partition function.
\end{abstract}

\subsection{\centering Introduction}
Littlewood-Richardson(LR) coefficients can be thought of as the structure constants for the ring of Symmetric functions (the Grothiendick ring of characters for $GL_n$). Infact we have the following definition of the LR coefficients:

$$ S_{\lambda} S_{\mu} = \sum_{\nu} c_{\lambda \mu}^{\nu} S_{\nu},$$
where $S_{\lambda},S_{\mu},S_{\nu}$ are the Schur polynomials corresponding to the highest weights $\lambda,\mu, \nu$ respectively, and the LR coefficient $c_{\lambda \mu}^{\nu}$ can be seen as the multiplicity of $S_{\nu}$ in the product decomposition of $S_{\lambda},S_{\mu}$.

 Kostka numbers $K_{\lambda \mu}$ are defined as the number of ways one can fill the boxes of the Young diagram of $\lambda$ with $\mu_1 \,1's, \mu_2 \,2's$ upto $\mu_n\,n's$, in such a way that the entries in each row are non decreasing and in each column are strictly increasing (A semistandard tableau on $\lambda$ of type $\mu$.)

In \cite{king2004stretched}, it is shown that Kostka numbers can be viewed as a special type of LR coefficients. To be precise, $K_{\lambda \mu}= c_{\sigma \lambda}^{\tau}$, where
\[
  \begin{cases}
    \tau_i = \mu_i + \mu_{1+1} + \cdots     \\
    \sigma_i = \mu_{i+1} + \mu_{i+2} + \cdots
  \end{cases}
  \quad for \quad i=1,2,\ldots
\]

In the course of this paper, we will prove certain necessary and sufficient condition for $K_{\lambda \mu} \neq 0$ (see Section 2, Proposition \ref{th:1}). The following is the main theorem of this paper, which will be proven in Section 3:

\begin{theorem}\label{th:0}
 $$S_{\lambda} S_{\mu} = \sum_{\xi \trianglelefteq \mu} K_{\mu \xi} (-1)^{\eta} \sum_{\sigma \in S_n/G(\xi)} S_{\eta(\lambda + \sigma(\xi) + \rho) -\rho} $$
 
where $G(\xi)$ is the stabalizer of $\xi$ in $S_n$, $\rho$ is the half sum of positive roots, $(-1)^{\eta}$ is the sign of $\eta \in S_n$ where $\eta \in S_n$ is the unique permutation which brings $\lambda + \sigma(\xi) + \rho$ to the positive Weyl chamber.
\end{theorem}

As a corollary to this theorem, we get the following representation of LR coefficients in terms of Kostka numbers:

\begin{corollary}
 With the above notation, we have that
  $$c_{\lambda \mu}^{\nu} = \sum_{\substack{\eta \in S_n\\\psi \trianglelefteq \mu  }} (-1)^{\eta} K_{\mu \psi}$$
 where  $\psi =\tau(\eta(\nu + \rho) - (\lambda +\rho))$, $\tau \in S_n$ is the permutation bringing the touple $(\eta(\nu + \rho) - (\lambda +\rho))$ in the Weyl chamber to the positive Weyl chamber.
\end{corollary}

Using this corollary, we will prove Stienberg's formula \cite{steinberg1961general}, by writing Kostka numbers in terms of Kostant partition functions ($\mathfrak{P}$):

$$c_{\lambda \mu}^{\nu} = \sum_{\sigma, \tau \in S_n} (-1)^{\sigma \tau} \mathfrak{P}(\sigma(\lambda +\rho) + \tau (\mu +\rho) - (\nu +2 \rho))$$

In Section 4, we give the algorithm which allows us to compute LR coefficents in polynomial time given that we have a black box that provides us with the values of Kostka numbers in polynomial time, hence showing that computing Kostka numbers and LR coefficents are in the same class of decision problems.

\subsection{\centering Preliminaries and Notation}
By a partition $\lambda$, we refer to a touple of non-negative integers $(\lambda_1, \lambda_2,\ldots, \lambda_n)$, where $\lambda_1 \geq \lambda_2 \geq \cdots \geq \lambda_n \geq 0$. Denote the content of $\lambda$ by $|\lambda| := \sum_i \lambda_i$. We use the term partitions and highest weight interchangably. Given partitions $\lambda_1,\lambda_2, \ldots , \lambda_k$, $\mathbf{size}(\lambda_1,\lambda_2, \ldots , \lambda_k)$ denotes the number of bits used in the description of this touple. Another way to look at a partition is via \textbf{Young tableau/diagram}. A Young tableau corresponding to a partition $\lambda$ (called a Young tableau of shape $\lambda$) is a collection of boxes, arranged in left justified rows, such that the first row has $\lambda_1$ boxes, second row has $\lambda_2$ boxes, and similarly the $n^{th}$ row had $\lambda_n$ boxes ( the number of boxes per row is weakly monotonically decreasing as $\lambda_1 \geq \lambda_2 \geq \cdots \geq \lambda_n \geq 0$). For example, $\lambda =(5,3,2)$ , the corresponding Young tableau is 
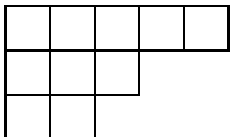
\begin{center}
\ydiagram{5,3,2}
\end{center}

A semistandard Young tableau of shape $\lambda$ is a Young diagram of shape $\lambda$ in which the boxes are filled with integers in such a way that the entries in each row are non decreasing and in each column are strictly increasing.

Kostka numbers $K_{\lambda \mu}$ are defined as the number of ways one can fill the boxes of the Young diagram of shape $\lambda$ with $\mu_1 \,1's, \mu_2 \,2's$ upto $\mu_n\,n's$ such that we get semistandard Young tableaus called semistandard tableau on $\lambda$ of type $\mu$). For $\lambda= (5,3,2)$ and $\mu = (4,3,3)$, the following are the only two semistandard Young tableau with shape $\lambda$ of type $\mu$:

\begin{center}
 \ytableaushort{11112,223,33} \hspace{2cm}
\ytableaushort{11113,222,33}

\end{center}
 which shows that $K_{(5,3,2) (4,3,3)} = 2$.
 
 We will hereby follow the notion that $\mathbf{x}$ would refer to the touple $(x_1,x_2,\ldots,x_n)$, and similarly $\mathbf{y}$ refers to $(y_1,y_2,\ldots,y_n)$. We recall that for $\lambda =(\lambda_1, \lambda_2,\ldots, \lambda_n)$, the Schur polynomial $S_{\lambda}$ is defined by the following ratio of alternating polynomials:
$$S_{\lambda}(\mathbf{x}) = \frac{\det (x_i^{\lambda_j+n-j})}{\det (x_i^{n-j})}.$$

The Schur polynomial $S_{\lambda}$ corresponds to the character of $\mathbf{S}^{\lambda}$, the Schur module which is the irreducible representation of $GL_{n}(\mathbb{C})$ with highest weight $\lambda$. By varying over all partitions with atmost $n$ parts ($n$ rows in the corresponding young diagram), Schur polynomials gives a basis for the ring of symmetric functions on $n$ variables($\mathbb{C}[x_1,x_2,\ldots,x_n]^{S_n}$).

LR coefficients are the structure constants in the ring of symmetric functions (the Grothiendick ring of characters for $GL_n$). For partitions $\lambda, \mu, \nu$, such that $|\nu| = |\lambda|+|\mu|$, we have that

$$ S_{\lambda} S_{\mu} = \sum_{\nu} c_{\lambda \mu}^{\nu} S_{\nu},$$

Equivalently, one can state the character (symmetric function) version with the notion that $\langle,\rangle$ is the Hall inner product on the ring $\mathbb{C}[x_1,x_2,\ldots,x_n]^{S_n}$.

$$c_{\lambda \mu}^{\nu} = \langle S_{\lambda} S_{\mu}, S_{\nu}\rangle.$$ 

Equivalently, one can define Kostka numbers as 
$$H_{\mu} = H_{\mu_1} H_{\mu_2} \ldots H_{\mu_n} = \sum_{\lambda} K_{\lambda \mu} S_{\lambda},$$

\noindent where $H_k$ correspont to the $k^{th}$ complete symmetric polynomial . We have the following generating function for $H_k$:

$$\prod_{i=1}^n \frac{1}{1-tx_i} = \sum_{k=0}^{\infty} t^k H_k(\mathbf{x}).$$


We make use of the following identity which can be seen in Appendix A of \cite{fulton1991first}.

\begin{theorem*}[Cauchy's Identity]
$$\prod_{1 \leq i,j \leq n}\frac{1}{ (1- x_i y_j)} = \sum_{\mu} S_{\mu}(\mathbf{x}) S_{\mu}(\mathbf{y}),$$
where the sum is over all partitions $\lambda$ with atmost $n$ terms.
\end{theorem*}

We prove the following Proposition, giving a criterion for non-vanishing of Kostka numbers that we would require in the main result (this can be seen as Exercise A.11 in \cite{fulton1991first}, but for better lack of reference, we give a complete proof here) :

\begin{proposition}\label{th:1}
 $K_{\mu \xi} \neq 0 $ \emph{iff} $|\xi| = |\mu|$ and $\xi_1 +\xi_2 + \cdots +\xi_i \leq \mu_1 +\mu_2 +\cdots + \mu_i$ for all $i\geq1$.
\end{proposition}

If $\trianglelefteq$ denotes the dominance ordering on partions of of same size, then Lemma \ref{th:1} is equivalent to saying $K_{\mu \xi} \neq 0 \iff \xi \trianglelefteq \mu$

\begin{proof}
 If $K_{\mu \xi} \neq 0 $, then there is a semistandard tableau on $\mu$ of type $\xi$. In a semistandard tableau, a box with `i' can occur in atmost the $i^{th}$ row (as columns are strictly increasing).  There are $\xi_1 \,1's, \xi_2 \,2's,\ldots \xi_i \,i's, $ which can atmost occur upto the $i^{th}$ row of $\mu$ which immediately gives us the inequality 
 $$\xi_1 +\xi_2 + \cdots +\xi_i \leq \mu_1 +\mu_2 +\cdots + \mu_i.$$
 Also, $|\xi| = |\mu|$ is a necessary condition for the tableau to be semistandard.
 
For the converse, we have that $\xi \trianglelefteq \mu$ is given, we will use induction on the number of rows of $\xi$, to show that this gives us a semistandard tableau on $\mu$ of type $\xi$. If there is only one row in $\xi$, then there is exactly one row in $\mu$ and they have same lenght, so the base case is true.

Assume that if there are $k-1$ rows in $\eta$ and $\eta \trianglelefteq \lambda$, then there is a semistandard tableau on $\lambda$ of type $\eta$. Let $\xi, \mu$ be such that $\xi \trianglelefteq \mu$, and the number of rows in $\xi$ is $k$. By the inequality, we know that there exists an $r$, such that $\mu_{r+1} < \xi_k \leq \mu_r$. Let $m$ be the number of rows in $\mu$, we fill the tableau of $\mu$ from the bottom, going right to left, filling atmost $\mu_m$ $k's$ in the $m^{th}$ row, atmost $\mu_{m-1}-\mu_m$ $k's$ in the $(m-1)^{th}$ row, and continuing like this (exhaustively filling right to left, then going to the row above), we reach the $r^{th}$ row, where the $k's$ get exhausted. Consider $\mu= (14,11,6,5,3)$ and $\xi = (9,8,8,7,7)$, then $\xi \trianglelefteq \mu$, and the following Young's diagram shows how we are filling the seven $5's$

\begin{center}

\ydiagram[*(white) 5]
{14+0,10+1,5+1,3+2,0+3}
*[*(lightgray)]{14,11,6,5,3}
\end{center}

The boxes in gray are the boxes still empty in the tableau after filling $5's$,and can be seen as $(14,10,5,3)$. In general, the empty boxes in the tableau of $\mu$ can be seen as $\mu' = (\mu_1, \mu_2, \ldots , \mu_{r-1}, \mu_r -(\xi_k -\mu_{r+1}), \mu_{r+2},\ldots ,\mu_m)$. We claim that $\xi' \trianglelefteq \mu'$, where $\xi'= (\xi_1, \xi_2,\ldots, \xi_{k-1})$, which has exactly $k-1$ rows. Hence we can use induction hypothesis to get a semistandard tableau on $\mu$ of type $\xi$ ( as the numbers to be filled are all $\leq k-1$, the filled boxes filled with $k$ would keep the tableau semistandard). For $1\leq i \leq r-1$, we get that $\xi_1 +\xi_2 + \cdots +\xi_i \leq \mu_1 +\mu_2 +\cdots + \mu_i$ as $\xi \trianglelefteq \mu$.  Note that $\xi_{s} \geq \xi_k$ for all $s$. Hence for all $s$, such that $r \leq s \leq k-1$, we get that 

$$\xi_1 +\xi_2 + \cdots +\xi_s + \xi_k \leq \xi_1 +\xi_2 + \cdots +\xi_s + \xi_{s+1} \leq \mu_1 +\mu_2 +\cdots + \mu_s + \mu_{s+1}$$

\noindent which gives us $\xi_1 +\xi_2 + \cdots +\xi_s \leq \mu_1 +\mu_2 +\cdots + (\mu_r +\mu_{r+1}-\xi_k) + \mu_{r+2} + \cdots \mu_s + \mu_{s+1}$, proving the claim.

\end{proof}

Let $\Phi^{+} \subset \Phi$ be a set of positive roots in a root system. The Kostant partition function $\mathfrak{P}$ for $\Phi^+$ is defined via the relation

$$\prod_{\alpha \in \Phi^+} \frac{1}{1-\mathbf{e}^{\alpha}}= \sum_{\omega} \mathfrak{P}(\omega) \mathbf{e}^{\omega}$$ 

Note that $\mathbf{e}$ is just a formal variable in this context, but with this convention $\mathbf{e}^a \mathbf{e}^b = \mathbf{e}^{a+b}$ . With this notation in picture, one can show that Kostka numbers can be expressed as a signed sum of Kostant partition function. If $\lambda, \mu$ are partitions of $n$, we have the following relation:

$$K_{\lambda \mu} = \sum_{\sigma \in S_n} (-1)^{\sigma} \mathfrak{P}(\sigma(\lambda + \rho) - (\mu +\rho))$$

We refer the interested reader to \cite{MR2284278} for a complete proof of the above equality.

The class $\# P$ is the class of functions $f : \cup _{m \in \mathbb{N}} \{ 0,1\} \rightarrow \mathbb{Z}_{\geq 0}$ for which there exists a polynomial time Turing machine $M$ and a polynomial $p$ such that $(\forall m \in \mathbb{N}), (\forall x \in \{0,1\}^m), f(x) = |\{ y \in \{0,1\}^{p(m)}$ such that $M$ accepts $(x,y)\}|$. Many counting problems like number of integer points in a polytope are in $\#P$. A counting problem $X\in \#P$ is $\#P-$complete if given a black box that provides solutions to instances  of $X$ in polynomial time, any problem in the class $\#P$ can be solved in polynomial time.

\subsection{\centering Main Theorem}

We give a proof of Theorem \ref{th:0}: 

\begin{proof}
 We take the formal product:
 
 \begin{align}
  S_{\lambda}(\mathbf{x})  \sum_{\mu} S_{\mu}(\mathbf{x}) S_{\mu}(\mathbf{y}) &= S_{\lambda}(\mathbf{x}) \prod_{1 \leq i,j \leq n}\frac{1}{ (1- x_i y_j)}\label{eq:0} \\ 
  &= \frac{\det (x_i^{\lambda_j+n-j})}{\det (x_i^{n-j})} \prod_{1 \leq i,j \leq n}\frac{1}{ (1- x_i y_j)}\\
  &= \frac{1}{\det (x_i^{n-j})} \det \left(\frac{x_i^{\lambda_j+n-j}}{\prod_{k=1}^n (1- x_i y_k)} \right)\label{eq:1}
 \end{align}
 consider the $ij^{th}$ entry in the above determinant:
 \begin{align*}
  \frac{x_i^{\lambda_j+n-j}}{\prod_{k=1}^n (1- x_i y_k)}  &= x_i^{\lambda_j+n-j} \sum_{r=0}^{\infty} x_i^r H_r(\mathbf{y})\\
  &= \sum_{r=0}^{\infty} x_i^{(\lambda_j+r)+n-j} H_r(\mathbf{y})\\
 \end{align*}
so \eqref{eq:1} becomes
\begin{equation} \label{eq:2}
\frac{\det \left(\sum_{k_{ij}=0}^{\infty} x_i^{(\lambda_j+k_{ij})+n-j} H_{k_{ij}}(\mathbf{y})\right)}{\det (x_i^{n-j})}
\end{equation}
We make a choice of the indices $k_{ij} = k_j$ ( specializing with this choice allows us to expand the determinant along the rows and take $H_{k_j}(\mathbf{y})$ common from them). With this choice, we can write the $i^{th}$ row as
$$\left[\sum_{k_{ij}=0}^{\infty} x_i^{(\lambda_j+k_{ij})+n-j} H_{k_{ij}}(\mathbf{y})\right]_i =\sum_{k_{j}=0}^{\infty} H_{k_{j}}(\mathbf{y}) [x_i^{(\lambda_j+k_{j})+n-j}]_i $$
this along with the fact that deteminant is an $n$ linear function, we get that \eqref{eq:2} becomes:

\begin{equation} \label{eq:3}
\sum_{(k_1,k_2,\ldots,k_n) \in \mathbb{Z}_{\geq 0}^n} \prod_{j=1}^n H_{k_{j}} \frac{\det (x_i^{(\lambda_j+k_{j})+n-j})}{\det (x_i^{n-j})}  
\end{equation}

Given $\mathbf{k}=(k_1,k_2,\ldots,k_n) \in \mathbb{Z}_{\geq 0}^n$, we define $G(\mathbf{k}) = \lbrace \sigma \in S_n | \sigma(\mathbf{k}) = \mathbf{k} \rbrace$ as the stabalizer of $\mathbf{k}$, and then $\sigma(\mathbf{k})$ where $\sigma \in S_n/G(\mathbf{k})$ ranges over all possible permutations of $\mathbf{k}$ without double counting. With this, we reindex the sum to be over partitions, so \eqref{eq:3} becomes
\begin{equation} \label{eq:4}
\sum_{\xi} H_{\xi}(\mathbf{y}) \sum_{\sigma \in S_n/G(\xi)} \frac{\det (x_i^{\lambda_j+\xi_{\sigma(j)}+n-j})}{\det (x_i^{n-j})}  
\end{equation}

If for any $k \neq j$, if $\lambda_j+\xi_{\sigma(j)}+n-j = \lambda_k+\xi_{\sigma(k)}+n-k$, then $\det (x_i^{\lambda_j+\xi_{\sigma(j)}+n-j}) =0$, so for the ones that remain, we have that $\lambda_j+\xi_{\sigma(j)}+n-j$ are distinct for all $j$, hence there exists a unique $\eta \in S_n$ such that $\eta(\lambda + \sigma(\xi) + \rho)-\rho$ is a dominant weight ( $\eta$ acts to bring the touple inside the positive Weyl chamber). The action of $\eta$ will bring about a sign change as well:

\begin{align}
 \sum_{\xi} H_{\xi}(\mathbf{y}) \sum_{\sigma \in S_n/G(\xi)} \frac{\det (x_i^{\lambda_j+\xi_{\sigma(j)}+n-j})}{\det (x_i^{n-j})}   &= \sum_{\xi} H_{\xi}(\mathbf{y}) \sum_{\sigma \in S_n/G(\xi)} (-1)^{\eta}S_{\eta(\lambda + \sigma(\xi) + \rho)-\rho}(\mathbf{x})\\
 &= \sum_{\xi} \sum_{\nu} K_{\nu \xi} S_{\nu}(\mathbf{y}) \sum_{\sigma \in S_n/G(\xi)} (-1)^{\eta} S_{\eta(\lambda + \sigma(\xi) + \rho)-\rho} (\mathbf{x}) \label{eq:5}
\end{align}
 By Proposition \ref{th:1}, we get that $K_{\nu \xi} \neq 0 $ iff $\xi \trianglelefteq \nu$, so \eqref{eq:5} becomes:
 \begin{equation}
   \sum_{\nu} S_{\nu}(\mathbf{y}) \left( \sum_{\xi \trianglelefteq \nu }K_{\nu \xi} \sum_{\sigma \in S_n/G(\xi)} (-1)^{\eta} S_{\eta(\lambda + \sigma(\xi) + \rho)-\rho}(\mathbf{x})\right)
 \end{equation}
 Comparing this with \eqref{eq:0} we get that the coefficient of $S_{\mu}(\mathbf{y})$ in both the equations must be the same, which gives us the following equality:
 $$  S_{\lambda}(\mathbf{x}) .S_{\mu}(\mathbf{x}) = \sum_{\xi \trianglelefteq \mu} K_{\mu \xi} (-1)^{\eta} \sum_{\sigma \in S_n/G(\xi)} S_{\eta(\lambda + \sigma(\xi) + \rho) -\rho} (\mathbf{x})$$
\end{proof}

\begin{corol*}
 With the above notation, we have that
  $$c_{\lambda \mu}^{\nu} = \sum_{\substack{\eta \in S_n\\\psi \trianglelefteq \mu  }} (-1)^{\eta} K_{\mu \psi}$$
 where  $\psi =\tau(\eta(\nu + \rho) - (\lambda +\rho))$, $\tau \in S_n$ is the permutation bringing the touple $(\eta(\nu + \rho) - (\lambda +\rho))$ in general position.
\end{corol*}

\begin{proof}
 We know from the theorem that 
 
 $$S_{\lambda}(\mathbf{x}) S_{\mu}(\mathbf{x}) = \sum_{\nu} c_{\lambda \mu}^{\nu} S_{\nu} (\mathbf{x})= \sum_{\xi \trianglelefteq \mu} K_{\mu \xi} (-1)^{\eta} \sum_{\sigma \in S_n/G(\xi)} S_{\eta(\lambda + \sigma(\xi) + \rho) -\rho}(\mathbf{x}) $$
 
 We solve for for $\xi$ in $\nu = \eta(\lambda + \sigma(\xi) + \rho) -\rho $ to get the coefficient of $S_{\nu}$. We get that $\sigma(\xi) = \eta^{-1}(\nu +\rho) - (\lambda +\rho)$, where $\sigma(\xi)$ is supposed to vary over all  permutations of $\xi$. We get that by varying $\eta \in S_n$, such that $\eta(\nu + \rho) - (\lambda +\rho)$ lies in the Weyl chamber, we can recover all the permutations of $\xi$ . There would be a unique choice of $\tau \in S_n$, such that $\psi =\tau(\eta(\nu + \rho) - (\lambda +\rho))$ is in the positive Weyl chamber, hence a dominant weight. With these notations, we get that the coefficient of $S_{\nu}$ is given by
 $$c_{\lambda \mu}^{\nu} = \sum_{\substack{\eta \in S_n\\\psi \trianglelefteq \mu  }} (-1)^{\eta} K_{\mu \psi}$$
\end{proof}

As an example, we can do the computation for $\lambda = (5,3,2),\mu=(4,3,3), \nu= (9,6,5)$ to see that the value of $\eta(\nu + \rho) - (\lambda +\rho)$ hence obtained are $(4,3,3),(5,4,1),(7,3,0),(9,1,0),$ $(7,5,-2),(9,3,-2)$. The last two are not in the Weyl chamber, $(4,3,3) \trianglelefteq \mu$, and none of the others are $\trianglelefteq \mu$, so the only contribution comes from $K_{(4,3,3)(4,3,3)} =1$, which tells us that $c_{(5,3,2)(4,3,3)}^{(9,6,5)}= 1$.

Another way to do the same computation, which gives us a polynomial time algorithm is to firstly find out all the partitions which are $\trianglelefteq \mu$, and in this case the only such partition is $(4,3,3)$. Next we see if $(4,3,3)$ can be written as $\eta(\nu + \rho) - (\lambda +\rho)$, which is true for $\eta$ being the trivial permutation, which tells us that the only contribution comes from $K_{(4,3,3)(4,3,3)} =1$, hence $c_{(5,3,2)(4,3,3)}^{(9,6,5)}= 1$.

\begin{theorem}[\textbf{Stienberg's formula}]
LR-coefficients can be given via the following formula:
$$c_{\lambda \mu}^{\nu} = \sum_{\sigma, \tau \in S_n} (-1)^{\sigma \tau} \mathfrak{P}(\sigma(\lambda +\rho) + \tau (\mu +\rho) - (\nu +2 \rho))$$

where $\mathfrak{P}$ is the Kostant partition function.
\end{theorem}

\begin{proof}
 If $\lambda, \mu$ are partitions of $n$, we know that

$$K_{\alpha \beta} = \sum_{\sigma \in S_n} (-1)^{\sigma} \mathfrak{P}(\sigma(\alpha + \rho) - (\beta +\rho))$$

Also the partition $\beta$ need not be ordered to define Kostka numbers. By the above corollary, we have that

$$c_{\lambda \mu}^{\nu} = \sum_{\eta \in S_n} (-1)^{\eta} K_{\mu  \tau(\eta(\nu + \rho) - (\lambda +\rho))}$$
if $\tau(\eta(\nu + \rho) - (\lambda +\rho)) \trianglelefteq \mu$, then the corresponding Kostka number is zero. Note that $\tau$ was chosen so that $\tau(\eta(\nu + \rho) - (\lambda +\rho))$ is in general position. We could make any choice of $\tau \in S_n$ and keep the formula invariant ( a bit more cumbersome to compute, maybe). We choose $\tau = \eta^{-1}$ for each of the terms, to get 

\begin{align*}
c_{\lambda \mu}^{\nu} &= \sum_{\tau \in S_n} (-1)^{\tau} K_{\mu  \nu + \rho - \tau(\lambda +\rho))}\\
&= \sum_{\tau \in S_n} (-1)^{\tau} \sum_{\sigma \in S_n} (-1)^{\sigma} \mathfrak{P}(\sigma(\mu + \rho) - (\nu + \rho - \tau(\lambda +\rho) +\rho) )\\
&= \sum_{\tau \in S_n} \sum_{\sigma \in S_n} (-1)^{\tau}(-1)^{\sigma} \mathfrak{P}(\sigma(\mu + \rho) + \tau(\lambda +\rho) - (\nu + 2\rho  ))\\
&= \sum_{\tau,\sigma \in S_n}  (-1)^{\tau \sigma} \mathfrak{P}(\sigma(\mu + \rho) + \tau(\lambda +\rho) - (\nu + 2\rho  ))
\end{align*}
 which proves the Stienberg's formula.
\end{proof}

\subsection{\centering Algorithm}

Given a black box which gives us the values of Kostka numbers $K_{\mu \psi}$ in polynomial time , we will give a polynomial time algorithm to compute the LR coefficients. This algorithm formalises the idea of the computation in previous section. We write the steps of the algorithm and then compute the complexity:

\begin{enumerate}
 \item Input  a positive integer $n$  and partitions in the form of an ordered $n-$ touple $\lambda, \mu, \nu $, and initialize the integer $c_{\lambda \mu}^{\nu}=0$. 
 \item List $\psi$ such that $\psi \trianglelefteq \mu$.
 \item For each $\psi$ in the list, check if $\psi =\tau(\eta(\nu + \rho) - (\lambda +\rho))$ for some $\eta, \tau \in S_n$:
 \begin{enumerate}
  \item Set $\psi^{(1)} := (\psi_1 + \lambda_1 +n-1,\psi_2 +  \lambda_1 +n-1, \ldots, \psi_n +  \lambda_1 +n-1)$ and $(\nu+ \rho)^{(1)} : = \nu+ \rho$
  \item If no entries of $\psi^{(1)}$ match with any entries in $(\nu + \rho)^{(1)}$, then $\psi \neq \tau(\eta(\nu + \rho) - (\lambda +\rho))$ for any $\eta, \tau \in S_n$.
  \item Otherwise say the $j^{th}$ entry $\psi_j + \lambda_1 +n-1$ is equal to the $k^{th}$ entry of $(\nu + \rho)$, namely $\nu_k +n -k$.
  \item Set $\psi^{(2)} := (\psi_1 + \lambda_2 +n-2,\psi_2 +  \lambda_1 +n-2, \ldots, \widehat{\psi_j + \lambda_2+n-2},\ldots, \psi_n +  \lambda_2 +n-2)$, where the $j^{th}$ entry has been omitted. Similarly omit the $k^{th}$ entry from $(\nu+ \rho)^{(1)}$ to get $(\nu+ \rho)^{(2)}$.
  \item Repeat Steps $(b),(c),(d)$ on $\psi^{(2)}, (\nu+ \rho)^{(2)}$, to get $\psi^{(3)}, (\nu+ \rho)^{(3)}$, and continue till either no entries match, or we reach $\psi^{(n+1)}, (\nu+ \rho)^{(n+1)}$ which are both empty, hence we stop. In this case, the positions of entries omitted in $(\nu+ \rho)$ gives us the permutation $\eta$.
 \end{enumerate}

 If \texttt{true} then add $(-1)^{\eta} K_{\mu \psi}$  to $c_{\lambda \mu}^{\nu}$. 
 \item Output $c_{\lambda \mu}^{\nu}$. 
\end{enumerate}

\textit{Step 2} can be done in $O(\mathbf{size}(\mu)^n)$ as it involves finding $\psi = (\psi_1,\psi_2, \ldots, \psi_n)$, and $\psi \trianglelefteq \mu$ gives us a rough bound that $\psi_i \leq |\mu|/i$ .

\textit{Step 3} has a loop, where we are trying to find if the given $\psi$ can be written as $(\nu + \rho) - (\lambda +\rho)$, note that $\psi^{(k)}$ are all preordered integer touples, so traversing and finding matching with $(\nu + \rho)^{(k)}$ takes $\sim k$ time, so in total it takes $O(n^2)$ time.

\textit{Step 1} and \textit{Step 4} are input and output, so they are $O(\mathbf{size}(\lambda,\mu, \nu)$.


\begin{conclusion*}
\normalfont
With this algorithm, we get that knowing Kostka numbers, one can generate LR coefficents in polynomial time. We already know via \cite{king2004stretched} that Kostka numbers can be expressed via LR coefficents in linear (polynomial) time, which tells us that they belong to the same class of decision problems. In \cite{narayanan2006complexity}, Hariharan shows that the computational complexity of computing either LR coefficients or Kostka numbers is $\# P-$complete. Our result reinforces this notion by showing polynomial time equivalence between the two.
\end{conclusion*}



\subsection{\centering Acknowledgement}

I express my gratitude to C.S. Rajan for the many valuable discussions while I was working on this paper, to Hariharan Narayan and Piyush Srivastava for informing me about the computational aspects of this problem, and other possible extensions. I also sincerely thanks to  Gaurav Bhatnagar, Manodeep Raha, Soumyadip Sahu ,Ashutosh Shankar for their suggestions. I'm especially thankful to Per Alexandersson for pointing out that my initial result could be used to prove Steinberg's formula. I'm extremely grateful to Ashoka University for warm hospitality and for providing a productive working atmosphere during my visit.

\bibliographystyle{alpha}

\begin{thebibliography}{3}

\bibitem[FH91]{fulton1991first}
William Fulton and Joe Harris.
\newblock {\em A First Course in Representation Theory}, volume 129.
\newblock Springer, 1991.

\bibitem[KTT04]{king2004stretched}
Ronald~C King, Christophe Tollu, and Fr{\'e}d{\'e}ric Toumazet.
\newblock Stretched {Littlewood-Richardson} and {Kostka} coefficients.
\newblock In {\em CRM Proceedings and Lecture Notes Vol. 34}, pages 99--112.
  American Mathematical Society, 2004.

\bibitem[Led06]{MR2284278}
Mathias Lederer.
\newblock On a formula for the {K}ostka numbers.
\newblock {\em Annals of Combinatorics}, 10(3):389--394, 2006.

\bibitem[Nar06]{narayanan2006complexity}
Hariharan Narayanan.
\newblock On the complexity of computing {Kostka} numbers and
  {Littlewood-Richardson} coefficients.
\newblock {\em Journal of Algebraic Combinatorics}, 24(3):347--354, 2006.

\bibitem[{S}te61]{steinberg1961general}
{R}obert {S}teinberg.
\newblock A general {Clebsch-Gordan} theorem.
\newblock {\em Bulletin of the American Mathematical Society}, 67(4):406--407,
  1961.
\end{thebibliography}

\end{document}